\documentclass[12pt]{amsart}
\usepackage{amssymb}
\def\pmatrix{\left(\begin{matrix}}
\def\endpmatrix{\end{matrix}\right)}

\def\same{{\underline{\hskip1.5cm}}}

\def\R{{\mathbb R}}
\def\P{{\mathbb P}}
\def\Z{{\mathbb Z}}
\def\C{{\mathbb C}}
\def\Q{{\mathbb Q}}
\def\H{{\mathcal H}}
\def\E{{\mathbb E}}

\def\p{\partial}

\def\t{\theta}
\def\T{\Theta}
\def\e{\varepsilon}
\def\de{\delta}
\def\ll{\lambda}

\def\A{{\mathcal A}}
\def\M{{\mathcal M}}
\def\X{{\mathcal X}}

\def\F{{\mathcal F}}

\def\Pic{\operatorname{Pic}}
\def\Sing{\operatorname{Sing}}
\def\Sp{{\operatorname{Sp}(2g,\Z)}}
\def\tt#1#2{{\t\left[\begin{matrix}{#1}\\{#2}\end{matrix}\right]}}
\def\AP{{\overline{\A_g}^P}}
\def\AV{{\overline{\A_g}^V}}

\newtheorem{OP}{Open Problem}
\theoremstyle{plain}
\newtheorem{thm}{Theorem}[section]

\newtheorem{prop}[thm]{Proposition}
\newtheorem{conj}[thm]{Conjecture}
\newtheorem{cor}[thm]{Corollary}

\theoremstyle{definition}
\newtheorem{df}[thm]{Definition}
\newtheorem{rem}[thm]{Remark}
\newtheorem{dsc}[thm]{}

\begin{document}
\title{Geometry of $\A_g$ and Its Compactifications}
\author{Samuel Grushevsky}
\address{Mathematics Department, Princeton University, Fine Hall,
Washington Road, Princeton, NJ 08544, USA}
\email{sam@math.princeton.edu}
\thanks{Research is supported in part by National Science Foundation under the grant DMS-05-55867.}
\begin{abstract}
In this survey we give a brief introduction to, and review the
progress made in the last decade in understanding the geometry of
the moduli spaces $\A_g$ of principally polarized abelian varieties
and its compactifications, concentrating on results obtained over $\C$.

This is an expanded and updated version of the talk given at the 2005
Summer Institute for Algebraic Geometry.
\end{abstract}

\maketitle
\tableofcontents
\section{Introduction}
In this survey we review the progress made in the last decade, the
current state of knowledge, and the open problems and possible
directions in the study of the geometry of the moduli spaces of
principally polarized abelian varieties and their compactifications,
primarily over the field of complex numbers.

We discuss the results on the geometric interpretation and
construction of compactifications; the study of the birational
geometry of $\A_g$, including nef and effective cones, and the
canonical model; the work on homology and Chow rings of $\A_g$;
constructions of special loci within $\A_g$ by using the geometry of
the theta divisor. Since the moduli space of curves $\M_g$ is
perhaps the best-studied moduli space, and is naturally a subvariety
of $\A_g$ via the Torelli map, we also draw analogies with the study
of $\M_g$ when appropriate. We mostly give references to the
original papers instead of complete proofs, but try to explain the
motivation for study, and some ideas leading to the proofs.

In this survey we focus primarily on the geometry of $\A_g$ rather
than that of individual abelian varieties, or of loci in $\A_g$
arising from special geometric constructions. In particular we do not
cover the exciting recent developments in understanding the geometry
of linear systems on one abelian variety (surveyed, for example, in
\cite{papo}). The more modular-theoretic aspects of the theory,
including a detailed study of subgroups of $\Sp$ and the associated
moduli spaces, are also not covered. Neither do we survey the
extensive literature on the problems of characterizing Jacobians of
Riemann surfaces within $\A_g$ (known as the Schottky problem),
including Krichever's recent proof \cite{trisecant} of Welters'
trisecant conjecture, of characterizing Prym varieties ---
characterized by the existence of a pair of quadrisecant planes in
\cite{grkr}, intermediate Jacobians of cubic threefolds ---
characterized by the existence of a triple point on the theta divisor \cite{cmfr},\cite{cm2}, etc. The history of the first two of these characterization problems is surveyed, for example, in \cite{taimanov}, from a more analytic viewpoint.

An earlier introduction and survey, with much more details on the cycles
on $\A_g$ and characteristic $p$, is \cite{vdgoo}. A survey giving
more details on the work on birational geometry of $\A_g$, including
the study of the non-principal polarizations, is \cite{husa1}. The
study of complex tori that are not necessarily algebraic is also
surveyed in \cite{debarrebook}. The book \cite{bila} contains a
wealth of information about complex abelian varieties, special loci,
theta functions, and moduli. The survey \cite{vdgsurvey} is focused
more on the theory of Siegel modular forms and related questions in
number theory.

{\bf Acknowledgements.} Those of the results surveyed in which I
participated have been obtained in collaboration with Cord
Erdenberger, Klaus Hulek, David Lehavi, and Riccardo Salvati Manni,
to all of whom I am grateful for the chance to investigate the
subject together. I am indebted to Klaus Hulek for detailed
discussions on the geometry of compactifications, and to Riccardo
Salvati Manni for discussions on the intricacies of the rings of
theta constants, and especially for explaining how Tai's technique
could be improved to yield small slopes (theorem \ref{minslope}). I
would like to thank Izzet Coskun, Gerard van der Geer, Klaus Hulek,
Nicholas Shepherd-Barron, and especially Riccardo Salvati Manni for reading drafts of this text very carefully, and for many useful suggestions and advice on content and presentation.


\section{Notations}
We start by defining the object of our discussion --- the moduli
space of principally polarized abelian varieties. Throughout the
text we will work over the base field $\C$, though many of the
results, especially the purely algebraic ones, carry over to
arbitrary base field. We make a few comments about the situation in positive characteristic in section \ref{charp}.

\begin{df}
Algebraically, an  {\it abelian variety} is  a projective
algebraic variety $A$, with the structure of an abelian group
on the set of its points, such that the group operations are morphisms $+:A\times A\to A$ and $-1:A\to A$.

A {\it polarization} on an abelian variety is the first Chern class of an ample line bundle $L$ on $A$, i.e. the polarization is $[L]:=c_1(L)\in H^2(A,\Z)\cap H^{1,1}(A,\C)$.
A polarization $[L]$ on an abelian variety $A$ is called {\it
principal} if its space of sections is one-dimensional, i.e.~if $h^0(A,L)=1$.
By abuse of notation, we will often think of any non-zero section of $L$ as the polarization, and
write $L$ instead of $[L]$.
\end{df}

\begin{df}
We denote by $\A_g$ the moduli space of principally polarized
abelian varieties (or ppavs for short) of dimension $g$, up to
isomorphisms preserving the principal polarization.
\end{df}

\begin{rem}
The moduli ``space'' $\A_g$, and its compactifications, to be defined below, are properly
to be thought of as stacks. However, for many considerations
thinking of $\A_g$ na\"\i vely as if it were a variety, or, more
carefully, an orbifold, suffices. To formally justify some of the work done on
$\A_g$ one needs to either work properly with a stack, or use the
fact that $\A_g$ admits finite covers (see below) that are actually
manifolds; often the stackiness does not present a problem. Note,
however, that any abelian variety has an involution $x\mapsto -x$,
and thus a general point of $\A_g$ in fact parametrizes an object
with an automorphism, so should be counted with multiplicity 1/2
as a stacky point.
\end{rem}

\begin{df}
We would now like to say that there exists a {\it universal family of
principally polarized abelian varieties} $\pi:\X_g\to\A_g$, with the
fiber over the point $[A]\in\A_g$ being the
variety $A$ itself. The existence of the universal family, even as a
stack, is not a trivial fact, and has to be proven and discussed in
more detail. However, for the base field being $\C$, the universal
family can be constructed as an explicit quotient (see below), and
thus we will be able to think of it very explicitly. Note that the
$\pm1$ involution is no longer trivial on $\X_g$, and thus a generic
point of $\X_g$ is in fact smooth. See \cite{fach} for a complete
discussion of the moduli stack and the universal family and of the
compactifications.
\end{df}

\begin{df}
There is a very important {\it Hodge vector bundle}
$\E:=\pi_*(\Omega^1_{\X_g/\A_g})$ on $\A_g$. This is just to say that
the fiber of the Hodge vector bundle over a point $[A]\in\A_g$ is the
$g$-dimensional space of holomorphic 1-forms on $A$. We denote by
$L:=\det \E$ the corresponding determinant Hodge line bundle.
\end{df}

$\A_g$ can be thought of algebraically, over any field. Let us now
give the analytic picture of it over $\C$. If the base field is
$\C$, the universal cover of any abelian variety is $\C^g$, and $A$
is given as a quotient of $\C^g$ by some action of $\pi_1(A)$. This
is to say that an abelian variety is a quotient of $\C^g$ by the
translations by elements of a full-rank lattice $\Lambda$, where by
a lattice we mean a subgroup of $\C^g$ (under addition) isomorphic
to $\Z^{2g}$, and a lattice is said to be of full rank if
$\Lambda\otimes_\Z\R=\C^g$. If we act on $\Lambda\subset\C^g$ by an element of $GL(g,\C)$, the quotient is going to be biholomorphic to the original one. Thus, up to biholomorphisms, any abelian variety is a quotient of $\C^g$ by a lattice the first $g$ generators of which are the unit vectors in all the directions. It turns out that (this is known as Riemann's bilinear relations) that for the quotient $\C^g/\Lambda$ to be a projective variety the other $g$ vectors must constitute a $g\times g$ matrix $\tau$ with a
positive-definite imaginary part. Such a complex matrix is called a {\it period matrix}.
\begin{df}
We denote by $\H_g$ the {\it Siegel upper half-space} --- the set of
all period matrices --- and for a period matrix $\tau\in\H_g$ denote
by $A_\tau:=\C^g/(\Z^g+\tau\Z^g)$ the corresponding abelian variety.
Notice that $\H_g$ is contractible.
\end{df}
Given a point $\tau\in\H_g$, there is a canonical choice of the
principal polarization on $A_\tau$.
\begin{df}
We define the {\it theta function} to be the holomorphic function of
$\tau\in\H_g$ and $z\in\C^g$, given by the following formula:
$$
  \t(\tau,z):=\sum\limits_{n\in\Z^g}\exp(\pi i(n^t\tau n+2n^tz)).
$$
The theta function is even in $z$, and automorphic in $z$ with respect to the lattice $\Z^g+\tau\Z^g$: for any $n,m\in\Z^g$ we have the transformation law
$$
  \t(\tau,z+\tau n+m)=\exp(-\pi i n^t\tau n-2\pi i n^t z) \t(\tau,z).
$$

Thus for a fixed $\tau$ the zero locus in $\C^g$ of the theta function, as a function in $z$, descends to a subvariety of $A_\tau$, which is called the {\it theta divisor $\T_\tau$}. This divisor then gives a principal
polarization on $A_\tau$, for which the theta function generates the space of sections.

\smallskip
The theta function satisfies the very important {\it heat equation}
$$
  \frac{\partial\t(\tau,z)}{\partial\tau_{jk}}=2\pi i(1+\de_{j,k})
  \frac{\partial^2\t(\tau,z)}{\partial z_j\partial z_k}.
$$
where $\de_{j,k}$ is the Kronecker symbol.
\end{df}
The map $\tau\mapsto A_\tau$ exhibits $\H_g$
as the universal cover of $\A_g$, and it is natural to ask what is the deck group of this
cover, i.e.~if the ppav $(A_\tau,\T_\tau)$ is isomorphic to
$(A_{\tau'},\T_{\tau'})$, how are $\tau$ and $\tau'$ related?

\begin{df}
It turns out that there is an action of ${\operatorname{Sp}(2g,\R)}$
on $\H_g$. If we think of ${\operatorname{Sp}(2g,\R)}$ as the group
of $2g\times 2g$ matrices written in the form of four $g\times g$
blocks such that the symplectic condition is
$$
  \pmatrix A&B\\ C&D\endpmatrix\pmatrix 0&1\\ -1&0\endpmatrix
  \pmatrix A&B\\ C&D\endpmatrix^t=\pmatrix 0&1\\ -1&0\endpmatrix,
$$
then the action is given by
$$
 \pmatrix A&B\\ C&D\endpmatrix\circ\tau:=(A\tau+B)(C\tau+D)^{-1}.
$$
A general element of ${\operatorname{Sp}(2g,\R)}$ does not map ppavs to
isomorphic ppavs; however, $\Sp$ does: if $\tau'=\gamma\circ\tau$ for some $\gamma\in \Sp$, then the ppav
$A_\tau$ is isomorphic to $A_{\tau'}$ (the map is $z\mapsto
(C\tau+D)z$), and it turns out that this is the only way $A_\tau$
and $A_{\tau'}$ can be isomorphic as ppavs, i.e.~that
$\A_g=\H_g/\Sp$.

We observe that $\dim\H_g=\dim\A_g=\frac{g(g+1)}{2}$. The universal
family $\X_g$ is then the quotient of $\H_g\times\C^g$ by the
semidirect product action of $\Sp\ltimes\Z^{2g}$ (where $\Z^{2g}$ acts on $\C^g$ by adding a lattice vector), and the fiber of the
Hodge bundle over $\tau$ is $\E|_\tau=H^{1,0}(A_\tau)=\C d z_1\oplus\ldots\oplus\C d z_g$. Notice that $\E$ lifts to a trivial vector bundle on $\H_g$, but it is not trivial on the quotient $\A_g$.
\end{df}
\begin{rem}
To be able to talk of $\A_g$ and $\X_g$ constructed as quotients of $\H_g$ and $\H_g\times\C^g$, respectively, as moduli spaces or fine moduli stacks, one needs to verify that the stabilizer of any point in $\A_g$ under the action of $\Sp$ (respectively, of any point in $\X_g$ under $\Sp \ltimes\Z^{2g}$) is finite. To show that this is the case for $\A_g$, note that any automorphism of an abelian variety can be lifted to a holomorphic map $\C^g\to\C^g$ of the universal covers fixing 0, which is of linear growth and thus linear. Then such a linear map must map the lattice to itself, and have determinant one (to be an isomorphism, and not finite-to-one), and then there can only be finitely many such maps. The proof for $\X_g$ is similar.
\end{rem}

\section{Modular forms and projective embeddings of $\A_g$}
The moduli space $\A_g$ is not compact. There are various
compactifications that one can define by studying what happens in
degenerating families of ppavs, and we devote the next section to
discussing these. Another approach to compactifying an algebraic
variety, however, is to construct an explicit embedding of it into a
projective space, and then compactify the image. For $\A_g$ this is
done by considering Siegel modular forms, which can be also thought of as functions on $\H_g$ with certain automorphy properties, or as some representations of $\Sp$, or as sections of certain bundles on $\A_g$. The study of
modular forms is a vast subject, of which we barely touch the tip
here --- it is exposed, for example, in the books \cite{igusabook},
\cite{freitagbook}. A comprehensive recent survey of Siegel modular
forms and of the questions arising already in dimension 2 is
\cite{vdgsurvey}.

\smallskip
Perhaps the simplest way to embed a variety into a projective space
is by sections of a very ample line bundle. Luckily, the Hodge
line bundle $L$ is actually ample on $\A_g$, though not very ample,
but for full generality it pays to consider the more general
situation.

In general any vector bundle $V$ on a variety $X$ can be lifted to its universal cover $\widetilde X$. If $\Pic(\widetilde X)=0$, then a section of $V$ lifts to a global vector-valued function on $\widetilde X$, which transforms appropriately under the action of $\pi_1(X)$ on $\widetilde X$. This is the concept of automorphic forms: studying sections of bundles on a variety as functions on the universal cover, subject to certain transformation rules.

\begin{df}
Given a subgroup $\Gamma\subset \Sp$ and a rational representation
$\rho:GL(g,\C) \to GL(W)$ for some vector space $W$, a
{\it $\rho$-valued modular form} is a holomorphic map $F:\H_g\to W$ such that
$$
 F(\gamma\circ\tau)=\rho(C\tau+D)\circ F(\tau)\qquad
 \forall\gamma=\pmatrix A&B\\ C&D\endpmatrix\in\Gamma,\ \forall\tau\in\H_g
$$
(where, as always, we write $\gamma$ as four $g\times g$ blocks), such that moreover for $g=1$ we require $F$ to be regular at the cusps of $\H_1/\Gamma$.

If $W=\C$, and the representation is $\rho(\gamma)=\det(C\tau+D)^k$,
then the modular form is called {\it a (scalar) weight $k$ modular
form for $\Gamma$.} It can be shown, by writing down the
transformation law for holomorphic 1-forms on $A_\tau$ under the
action of $\Sp$ on $\H_g$, that the Hodge vector bundle $\E$
is in fact the bundle of modular forms for the standard (identity) representation, and thus $L$ is the bundle of (scalar) modular forms of weight 1.
\end{df}
It is hard to construct scalar modular forms of small weight for the
entire group $\Sp$. However, one can use the theta function to
construct modular forms for subgroups.
\begin{df}
For any $m\ge 2$ and any $\e,\de\in(\frac{1}{m}\Z/\Z)^g$ the {\it
level $m$ theta function with characteristics $[\e,\de]$} is defined
as
$$
\begin{matrix}
  \tt{\e}{\de}(\tau,z):=&\sum\limits_{n\in\Z^g}\exp(\pi i ((n+\e)^t
  \tau (n+\e)+2(n+\e)^t(z+\de)))\\
  &=\exp(\pi i(\e^t\tau\e+2\e^t(z+\de)))\t(\tau,z+\tau\e+\de).
\end{matrix}
$$
As a function of $z$, the level $m$ theta function is a section of
the theta bundle translated by the corresponding point of order $m$,
and thus $\tt\e\de(\tau,z)^m$ is a section of the
bundle $m\T_\tau$ on $A_\tau$ for all $\e,\de$. The space $H^0(A_\tau,m\T_\tau)$ is
$m^g$-dimensional, with the basis given by {\it theta functions of
order $m$}: for $\e\in(\frac{1}{m}\Z/\Z)^g$ these are defined as
$$
 \T[\e](\tau,z):=\tt\e0(m\tau,mz)
$$
\begin{rem}
To see that $\T[\e]$ is a section of the bundle $m\T_\tau$ on $A_\tau$, note that for a general ppav we have $H^2(A_\tau,\C)=\C\T_\tau$. Now compute the top power of the divisor of $\T[\e]$ on $A_\tau$, using $\T_\tau^g=g!$. Indeed, the multiplication by $m$ map has degree $m^{2g}$ on $A_\tau=\C^g/(m\Z^g+m\tau\Z^g)\cong\C^g/(\Z^g+\tau\Z^g)$, which is a degree $m^g$
cover of $A_{m\tau}=\C^g/(\Z^g+m\tau\Z^g)$, and thus the top power of the divisor of $\T[\e]$ on $A_\tau$ is $m^gg!$.
\end{rem}

The value of the (level or order) theta function at $z=0$ is called
the associated (level or order) {\it theta constant}. As a function
of $\tau$ for fixed $\e,\de$, the {\it order} $m$ theta constant is
a modular form of weight 1/2 with respect to the finite index subgroup
$\Gamma(m,2m)\subset\Sp$ (normal for $m$ even), defined as follows
in two steps:
$$
\begin{aligned}[rl]
  &\Gamma(m):=&\left\lbrace\gamma=\pmatrix A&B\\
  C&D\endpmatrix \in\Sp\, \right|\left.\, \gamma\equiv\pmatrix 1&0 \\
  0&1\endpmatrix\ {\rm mod}\ m\right\rbrace\\
  &\Gamma(m,2m):=&\left\lbrace \gamma\in\Gamma(m)\ |
  \ {\rm diag} (A^tB)\equiv{\rm diag}(C^tD)\equiv0\ {\rm mod}\ 2m\right\rbrace
\end{aligned}
$$
The {\it level} $m$ theta constant is also a modular form, also of
weight 1/2, with respect to the (smaller) group $\Gamma(m^2,2m^2)$.
\end{df}
\begin{rem}
Notice a peculiar feature of theta functions: as {\it
functions of $z$}, $m$'th powers of the level theta functions are
sections of the same bundle, $m\T$, on a fixed abelian variety, as
the theta functions of order $m$. However, theta {\it constants} of
any order or level are all of weight 1/2, with respect to the
appropriate level subgroups.
\end{rem}

\begin{df}
We call the quotient $\A_g(m,2m):=\H_g/\Gamma(m,2m)$ the {\it
level moduli space of ppavs} --- this is a finite cover of
$\A_g$. The subgroup $\Gamma(m,2m)\subset\Sp$ is normal if and only if $m$ is even, and in this case the cover is Galois.
Since all theta constants of order $m$
are sections of $\frac12 L$ on $\A_g(m,2m)$, we can use them to
define the {\it theta constant map}
$$
  \begin{aligned}
  Th_m:\A_g(m,2m)&\dashrightarrow\P^{m^g-1}\\
  [\tau]&\mapsto\lbrace\Theta[\epsilon](\tau,0)\rbrace_{
  {\rm all}\ \e\in (\frac{1}{m}\Z/\Z)^{g}}
  \end{aligned}
$$
\end{df}
A priori this is just a rational map, but the main result about it
is
\begin{thm}[Igusa for $m=4r^2$, Mumford for $m\ge4$, Salvati Manni for $m\ge3$; see \cite{igusabook},\cite{bila}]
$Th_m$ is an embedding for all $m\ge 3$.
\end{thm}

Algebraically this theorem says that the bundle $\frac12 L$ is
very ample on $\A_g(m,2m)$, which implies that a sufficiently high
power of $L$ is very ample on $\A_g$, and so $L$ is ample on
$\A_g$. This can be also checked directly by computing the curvature of the natural metric on $L$ and checking that it is positive.

The map $Th_2$ is known to be generically injective, and believed to be in fact an embedding --- see \cite{smlevel2}. It can in fact be shown that for $m=2k>2$ the level moduli space $\A_g(m,2m)$ (or in fact $\A_g(m)$ for any $m\ge 3$) is a smooth variety, i.e.~that the group $\Gamma_g(m,2m)$ acts freely on $\H_g$. Thus the orbifold $\A_g$ has a global manifold cover of a finite degree, which often allows one to work rigorously on the orbifold $\A_g$ by passing to the level cover.

\begin{rem}
Taking the closure of the image $Th_m(\A_g(m,2m))$ in $\P^{m^g-1}$ defines a compactification of the moduli space. It turns out that modular forms
extend to the Satake compactification (which we define in the next
section). Igusa used theta functions to study the fiber of $Th_m$ over the boundary, and showed that for $m>4$ it consists of more than one point (he computed the number of points for $m=4r^2$, and bounded it below for other $m$), while the map $Th_4$ is injective on the boundary of the Satake compactification as well. However, for $g\ge 6$ the map $Th_4$ is not an embedding of the Satake compactification --- the inverse is not regular near the boundary. The fact that there exist modular forms that are not polynomial in theta constants, and the relation of the analytic structure near the boundary of $Th_m(\A_g(m,2m))$ with the analytic structure of the Satake boundary are considered in \cite{iggen}, \cite{igdet}, \cite{smring}, \cite{smnotimm}.
\end{rem}

\begin{dsc}[{\bf Vector-valued modular forms}]

The above discussion tells us that the line bundle of scalar modular
forms is ample on $\A_g$. What about vector-valued modular forms?
This is some kind of ampleness question for a vector bundle. Let us
see what happens if $\rho:GL(g,\C)\to GL(\C^g)$ is the
standard representation $std$ tensored with a power of the $det$
(i.e.~a power of $L$).

It can be shown that the $z$-gradients at zero of order $m$ theta
functions
$$
 \left. \operatorname{grad}_z\Theta[\e](\tau,z)\right|_{z=0},
$$
are $std\otimes det^{1/2}$-valued modular forms for $\Gamma(m,2m)$.
Varying $\e$ one gets different modular forms, and thus for $m>2$ we can
define the map
$$
\begin{aligned}
  \Phi_m:\ \A_g(m,2m)&\dashrightarrow G(g,m^g)\\
  \tau&\mapsto\left\lbrace\operatorname{grad}_z\Theta[\e](\tau,z)|_{z=0}
  \right\rbrace_{{\rm all}\ \e\in (\frac{1}{m}\Z/\Z)^{g}},
\end{aligned}
$$
where $G(g,m^{g})$ denotes the Grassmannian of $g$-dimensional
subspaces of $\C^{m^{g}}$ (a priori it is a map to $Mat_{g\times
m^{g}}(\C)$, but it turns out \cite{smrank} that the rank of the
image matrix is always $g$). Notice that all theta functions of order 2 are even in $z$, and thus the map $\Phi_2$ is undefined.
\end{dsc}
\begin{thm}[--- and Salvati Manni; \cite{grsm1} for $m=4$, \cite{grsm3} for $m=4k>4$]\label{phiembed}
If the level $m=4k>4$, then the map $\Phi_m$ is an embedding, while $\Phi_4$ is generically
injective for $m=4$ (though we actually believe to be an embedding as well).
\end{thm}
The condition that $m$ is divisible by 4 is likely technical, but our proof, which deduces the injectivity of $\Phi_m$ from the injectivity of $Th_{m/2}$ and $Th_m$, uses it. Note also that one can consider the gradients at zero of theta functions of level $m$, but this does not give any new information.
\begin{rem}
This implies that the vector bundle of $std\otimes det^{1/2}$-valued
modular forms is very ample on $\A_g(m,2m)$ in some sense (it can be shown that the space of such modular forms is generated by gradients of theta functions). This
theorem has a geometric interpretation, and is related to classical
algebraic geometry. Indeed, on any ppav $A_\tau$ the line bundle
$\T_\tau$ is ample, and $m\T_\tau$ is very ample for $m\ge 3$ (this
fact is known as the Lefschetz theorem). For any characteristic of level $m$ the function $\tt\e\de(z)^m$ is a section of $|m\T_\tau|$. It can be shown that the space of sections $H^0(A_\tau,m\T_\tau)$ is generated by these $m$'th powers. $Th_m(\tau)$ is then the image of the origin in the corresponding embedding
$F:A_\tau\hookrightarrow \P^{m^{2g}-1}$. Instead of taking
$F(0)$, one can take the the differential $dF(0)$, which is exactly
$\Phi_m(\tau)$.

Given a plane quartic, its bitangent lines are in one-to-one correspondence with square roots of the canonical bundle with one section, i.e.~ with odd level $2$ theta constants, and
this is what the map $\Phi_2$ is for the corresponding Jacobian in
$\M_3\subset\A_3$. In \cite{cs1} Caporaso and Sernesi show that a
plane quartic is generically determined by its bitangents, in
\cite{cs2} they generalize this to higher genus curves, and in
\cite{lehavi} Lehavi explicitly reconstructs quartics from their
bitangents. Our result is almost a generalization of all these from
curves to ppavs (though not quite: there are some issues with
symmetrizing and projectivizing that we cannot deal with for
$\A_g$), and it is also a step towards better understanding the
rings of vector-valued modular forms and to perhaps answering an old
question of Weil, essentially on the relation of the maps $\det\Phi_{m}$ and
$Th_m$. We refer to \cite{fa}, \cite{igweil} and \cite{smnonzero}
for more details on the problem and past results; we used the above
framework to further investigate this with Salvati Manni in
\cite{grsm2}.
\end{rem}

\section{Degeneration: compactifications of $\A_g$}
In the previous section we constructed explicit projective
embeddings of level covers of $\A_g$, which thus naturally induce
some compactifications. We will now proceed to construct abstractly
compactifications of $\A_g$ and understand their geometry --- their
relation to the ones obtained from projective embeddings is still
not entirely clear. The discussion we present is necessarily greatly
simplified --- we refer to \cite{fach} for the complete details in
full generality, and also to \cite{amrt}, \cite{namikawa},\cite{alexeev},\cite{alna}, \cite{huleknotes},\cite{olsson} for more
comprehensive explanations and the intuition about toroidal
compactifications. A more detailed discussion of the explicit
boundary geometry, especially for $g=2$, can also be found in the
book \cite{hubook} and the survey \cite{husa1}, while the original
constructions are given in \cite{mumfordcomp}.

The Siegel space $\H_g$ is not compact --- the entries of a period
matrix $\tau$ can tend to infinity, or $\operatorname{Im}\tau$ can
become degenerate instead of being positive definite. It can be
shown that the action of $\Sp$ can conjugate the second kind of
degeneration into the first kind of degeneration --- so the only
degeneration one needs to consider in working with $\A_g$ is when
the entries of the period matrix grow unboundedly (however, as we will see later, to construct the toroidal compactification properly, one should rather consider matrices with positive semidefinite imaginary part).

To compactify $\A_g$ we need to attach some boundary points as
limits of degenerating families; it would also be nice to have some
geometric objects that are degenerations of abelian varieties
correspond to the extra points we add as the boundary. There are two
possible approaches.

{\bf Approach 1:} we take $[\tau]\in\A_{g-1}$ as the limit of the
degenerating family  $\lim\limits_{t\to\infty} \pmatrix it&w\\
w^t&\tau\endpmatrix$ (where $w\in\C^{g-1}$ and $\tau\in\A_{g-1}$ are
fixed), i.e.~we add $\A_{g-1}$ as a boundary component. This means
that the boundary is going to be high codimension and very singular.
However, the good thing is that when we consider more complicated
degenerations, the choice of what to do is natural. Indeed, we can
set for example
$$
 \lim\limits_{t_1,t_2\to\infty} \pmatrix it_1&x&w_1\\ x&it_2 &w_2\\
 w_1^t&w_2^t&\tau\endpmatrix=[\tau]\in\A_{g-2}
$$
(recall that the imaginary part of a period matrix is
positive-definite, so this is the way the degeneration has to look).
\begin{df}
The object we get as the result is called the {\it Satake}, or {\it Baily-Borel}, or {\it minimal, compactification of $\A_g$}. As a set, it is
$$
  \A_g^S:=\A_g\sqcup\A_{g-1}\sqcup\ldots\sqcup\A_{1}\sqcup\A_0,
$$
and much more work is necessary to properly describe the analytic and algebraic structure near the boundary.
It can be seen that modular forms extend to $\A_g^S$, i.e. that the bundle $L$ extends to $\A_g^S$ as a line bundle. The extension of theta constants to the level Satake compactification can be computed directly:
$$
  \lim\limits_{t\to\infty} \Theta[\e_1\, \e_2]\pmatrix
  it&w\\ w^t&\tau\endpmatrix=
\left(\lim\limits_{t\to\infty} \Theta[\e_1](it)\right) \Theta[\e_2](\tau)=
  \delta_{\e_1,0}\Theta[\e_2](\tau),
$$
i.e.~the extension is zero if $\e_1\ne0$. Thus the map $Th_m$ extends to $\A_g^S(m,2m)$.
\end{df}
The space $\A_g^S$ is highly singular, and the boundary points
represent lower-dimensional ppavs, which of course are not
degenerations of $g$-dimensional ppavs, so let us try to get a
different compactification.

{\bf Approach 2:} We say that $\lim\limits_{t\to\infty}\pmatrix it&w\\
w^t&\tau\endpmatrix$ is the pair $(\tau,w)$. The vector $w$ is only
defined up to $\tau\Z^{g-1}+\Z^{g-1}$ (we can act by the symplectic
group, preserving the one infinity in the period matrix), i.e.~we
have $w\in A_\tau$, and so can think of the pair
$(\tau,w)\in\X_{g-1}$ as a point in the universal family. Note,
however, that if $A_\tau$ has an automorphism $\sigma$ (and all
ppavs have involution $\pm 1$), then the points $\tau,w$ and
$\tau,\sigma(w)$ would define the same semiabelian object.
\begin{df}
The object we get by adding all of these boundary points is called
the {\it partial compactification of $\A_g$}. Set-theoretically it is
$$
 \A_g^*:=\A_g\sqcup\X_{g-1}/\pm 1.
$$
$\A_g^*$ is the blowup of the partial Satake compactification
$\A_g\sqcup\A_{g-1}$ along the boundary.
\end{df}
\begin{dsc}[{\bf Rank one semiabelian varieties}]

The boundary of $\A_g^*$ is codimension one; its points represent
(torus rank one) {\it semiabelian} varieties, which are defined as
follows: given $(\tau,w)\in\X_{g-1}$ compactify the
$\C^*$-extension
\begin{equation}\label{semiab}
  1\to\C^*\to G\to A_\tau\to0
\end{equation}
to a $\P^1$-bundle $\tilde G$, by adding 0 and $\infty$ sections,
and then identify the 0 and $\infty$ sections with a shift by $w\in
A_\tau$, getting a non-normal variety $\bar G:=\tilde G/(x,0)\sim (x+w,\infty)$. The principal polarization on such a semiabelian variety is
a codimension one subvariety of $\bar G$, which intersects the
zero section of the $\P^1$-bundle $\tilde G$ in the theta divisor of
$A_\tau$, and is globally a blowup of a section of $\tilde G$ with
center $\T_{A_\tau}\cap t_w\T_{A_\tau}$ ($t_w$ denotes the
translation by $w$). The existence of such a subvariety determines
the extension in (\ref{semiab}) uniquely --- it depends on $\tau$
and $w$.
\end{dsc}

\smallskip
No choice is involved in the construction of $\A_g^*$, but it is
still not compact. How can we extend it to an actual
compactification, i.e.~what should for example be the limit
$$
  \lim\limits_{t_1,t_2\to\infty}\pmatrix it_1&x&w_1\\ x&it_2&w_2\\
  w_1^t& w_2^t&\tau\endpmatrix?
$$
We can certainly keep track of $(\tau,w_1,w_2)\in \X_{g-2}^{\times
2\, ({\rm fiberwise})}$, but if we want this type of degenerations to
form a codimension two stratum in the compactification --- after all,
we have two entries of the period matrix degenerating --- we need one
more piece of data, and that is $x$. The problem is that $x$ may also
go to $i\infty$, and may change when we conjugate the period matrix
by elements of $\Sp$ while leaving the two infinities intact. Thus to
keep track of this extra coordinate properly (and to do this in
general for higher codimension generations) we need to make a choice
of a so-called cone decomposition. We now give an idea of what this
entails, and encourage the reader to learn the theory properly by
looking at \cite{amrt},\cite{namikawa},\cite{fach},
\cite{hubook},\cite{huleknotes},\cite{alexeev},\cite{olsson} and references therein.
\begin{dsc}[Cone decomposition]
Instead of making the entries of the period matrix go to infinity, we
would now rather think of the imaginary part becoming positive {\it
semi}definite. Fix generators $x_1,\ldots,x_g$ of $\Z^g$, and think
of the space $Sym_2(\Z^g)$ of integer-valued bilinear forms on
$\Z^g$. Identifying this with the space of quadratic forms, it is a
finite-dimensional free $\Z$-module generated by $x_i^2$ and
$2x_ix_j$ for $i\le j$. Denote by $\overline{C}(\Z^g)$ the $\R_{\ge
0}$-span of the positive semidefinite quadratic rational forms on
$\Z^g$, i.e.~$\overline{C}(\Z^g)$ is the cone generated by positive
semidefinite $g\times g$ rational matrices. All of this is the data
used to understand the orbits of the $\Sp$ action on the boundary of
$\H_g$, i.e.~on the set of symmetric matrices with positive
semidefinite imaginary part.

What we would now need is to somehow have local ``coordinates'' on
$\overline{C}(\Z^g)$, in which we would be able to keep track of the
degeneration happening. Doing so globally is impossible since
$\overline{C}(\Z^g)$ is not finitely generated. Thus what we need to
do is decompose it into infinitely many finitely-generated
polyhedral cones, i.e.~each cone should be a finite span $\R_{\ge
0}q_1+\ldots +\R_{\ge 0}q_k$, where $q_i\in Sym_2(\Z^g)$ are semipositive definite, and when
two cones intersect, they should intersect along a face. Moreover,
note that the natural action $GL(g,\Z):\Z^g$ extends to an action on
$\overline{C}(\Z^g)$, and thus it is natural to ask for our cone
decomposition to be invariant under this $GL(g,\Z)$ action. There
may of course exist different cone decompositions (each encoded by a
finite amount of data, though, as the cones in it would fall into
finitely many $GL(g,\Z)$-orbits), and choosing different ones yields
different toroidal compactification.
\end{dsc}
\begin{df}
The names for some common choices of the cone decompositions and the
corresponding toroidal compactifications are the following
(unfortunately it seems that defining and discussing the precise
construction of each of these would be quite long --- the readers
interested in this are advised to read more comprehensive sources
listed above):

The {\it perfect cone, {\rm also called} first Voronoi compactification} $\AP$.

The {\it second Voronoi compactification $\AV$}.

The {\it Igusa compactification $\overline{\A_g}^{\rm Igusa}$},
which is the monoidal blowup of the Satake compactification along
the boundary, corresponding to the central cone decomposition.
\end{df}

It was shown by Namikawa \cite{namikawa} that the Torelli embedding
$\M_g\hookrightarrow\A_g$  extends to a map (no longer an embedding)
$\overline{\M_g}\to\AV$ of the Deligne-Mumford compactification.

{\bf Example:} For genus 2 all the toroidal compactifications we
mentioned above coincide. They are defined by considering the
polyhedral cone
$$
  \sigma:=\R_{\ge 0}\begin{pmatrix} 1&0\\ 0&0\end{pmatrix}+
  \R_{\ge 0}\begin{pmatrix} 1&1\\ 1&1\end{pmatrix}+
  \R_{\ge 0}\begin{pmatrix} 0&0\\ 0&1\end{pmatrix}\subset
  \overline{C}(\Z^2),
$$
(notice that all generators are indeed degenerate), and the cone
decomposition of $\overline{C}(\Z^2)$ is obtained by taking the
$GL(2,\Z)$ orbits of $\sigma$ and of its faces.

\begin{rem}
All toroidal compactifications of $\A_g$ admit a contracting
morphism to $\A_g^S$. We remark, however, that the stratum over $\A_{g-i}\subset\A_g^S$ is in general very complicated. Even the dimension of the preimage of $\A_{g-i}$ in $\overline{\A_g}$ depends on the choice of the compactification: for example the stratum of $\AP$ lying over $\A_{g-i}$ always has codimension $i$, while already the preimage of $\A_0$ under the map $\overline{\A_4}^V\to\A_4^S$ is a divisor.
\end{rem}

It is natural to ask if boundary
points of a compactification of $\A_g$ have a geometric interpretation; do they parameterize some
degenerate objects that live in a universal family? For the case of
$\AV$, the answer to these questions was recently shown to be
positive:
\begin{thm}[Alexeev \cite{alexeev}]
The second Voronoi compactification $\AV$ is an irreducible
component of a functorial compactification of $\A_g$, i.e.~of some
``natural'' compactification from the point of view of moduli
theory, over which the universal family exists. Thus the boundary
points of $\AV$ represent geometric objects, and $\AV$ is
projective.
\end{thm}
\begin{thm}[Olsson \cite{olsson}]
Within the functorial compactification $\AV$ is distinguished as the
component parameterizing log smooth objects.
\end{thm}
In view of this theorem, and especially since it is still not even
known whether there is a universal family over $\AP$ or any other
toroidal compactification, one may ask whether $\AV$ is then the
``natural'' choice of a toroidal compactification, or whether any
other toroidal compactifications are singled out by some geometric
constructions? In the next section we will discuss why $\AP$ is also
very important. Meanwhile, there is another naturally singled out
compactification, though it may be one of those that we have defined
above.
\begin{OP}
Which compactification does the map $\Phi_m$ from theorem \ref{phiembed} induce, i.e.~what is
the structure of the closure of the image $\Phi_m(\A_g(m,2m))$, for $m=4k$?
\end{OP}
\begin{rem}
It can be shown by studying the degenerations of theta functions
directly that $\Phi_m$ extends to an embedding of $\A_g^*(m,2m)$ for $m=4k$. Since the Hodge vector bundle and its determinant line bundle extend as bundles to any toroidal compactification \cite{mumhirz}, the gradients of theta functions extend to the boundary of any toroidal compactification. However, the map $\Phi_m$ may not be defined on the boundary if the gradients no longer span a $g$-dimensional space, and injectivity seems very hard to deal with. We certainly get
some blowup of $\A_g^S$, since essentially we are somehow resolving
the singularities of $\A_g^S$ by taking derivatives of modular forms, but it is not even clear if the induced compactification is toroidal.

\smallskip
One can also ask what happens for maps induced by vector-valued
modular forms for representations of $GL(g,\C)$ other than
$std\otimes det^{1/2}$, but this currently seems to be entirely out
of reach: while we can hope to understand the degeneration of the
polarization and thus of theta functions, it is not clear how to
understand the extensions of general modular forms.
\end{rem}

\section{Birational geometry: divisors on $\A_g$}\label{biratgeom}
In this section we discuss the recent progress and the open questions
in the study of the birational geometry of $\A_g$ and its
compactifications. We give the description of the nef cone of
$\A_g^*$ (and of $\A_4^V$), due to Hulek and Sankaran; of the nef
cone of $\AP$, due to Shepherd-Barron, and the possible approaches
and known results about the effective cone. We also draw comparisons
with moduli of curves.

\smallskip 
It is a by now classical result of Borel in group cohomology saying
that $h^2(\Sp)=1$ for $g\ge 3$. Since $\Sp$ is the universal covering group for $\A_g$, and $\H_g$ is contractible, so that $NS(\H_g)=0$, this shows that the Neron-Severi group of $\A_g$ is one-dimensional. Moreover, one can show that a compactification of (an appropriate smooth level cover of) $\A_g$ is simply connected, or that such a compactification has no global holomorphic 1-forms. It then follows that for such a compactification the Picard group and the Neron-Severi group coincide (i.e. that the only numerically trivial bundle is the trivial bundle).

The space $\A_g$ is a stack, and has finite quotient singularities, so one can only talk of $\Q$-divisor on it. The classical results of Borel et al yield then $\Pic_\Q(\A_g)=\Q L$ (in fact for all $g$). Since the boundary divisor of $\A_g^*$ is irreducible, it follows that $\Pic_\Q(\A_g^*)=\Q L\oplus\Q D$. It can in fact be shown that the boundary divisor is also irreducible on $\AP$, so that it follows that $\Pic_\Q(\AP)=\Pic_\Q(\A_g^*)=\Q L\oplus\Q D$, while in general
$\Pic_\Q(\AV)$ is higher-dimensional.
\begin{df}
Recall that a divisor (we always talk about $\Q$-divisors, since we are on an orbifold/stack) is called
{\it ample} if on any subvariety (including the variety itself) its
top power is positive; a divisor is called {\it nef} (numerically
effective) if it intersects all curves non-negatively; and a divisor
is called {\it effective} if it is a positive linear combination of
codimension one subvarieties.

For a divisor $E=aL-bD\in\Pic_\Q(\A_g^*)=\Pic_\Q(\AP)$ we call the ratio
$s(E):=a/b$ {\it the slope of $E$}; if $E$ is (the closure in $\A_g^*$ or $\AP$ of) the zero locus in $\A_g$ of a
modular form, then the slope is the weight of the modular form
divided by the generic vanishing order on the boundary.
\end{df}
The sets of effective/nef/ample $\Q$-divisors form respectively the cones
$Eff/Nef/Amp$, which are important invariants. Since the group
$\Pic_\Q(\AP)=\Pic_\Q(\A_g^*)$ is two-dimensional for $g>1$, the slopes
of the boundaries of the cone (which we then call the slope of the cone, denoted
$s(Eff(\A_g^*)$, etc.) determine the cone, and computing these
cones may be more amenable than, say, for $\overline{\M_g}$, where the
Picard group is higher dimensional, and though there has been
significant progress in understanding the nef cone \cite{gkm} and the
minimal {\it slope} of the effective cone of $\overline{\M_g}$ (reviewed in \cite{farkassurvey})
the effective {\it cone} of $\overline{\M_g}$ is completely
unknown.
\begin{df}
For birational geometry it is especially important to know whether
the canonical class is ample, effective, or neither. The {\it
Kodaira dimension} of a variety $X$ is a number $\kappa$ such that
$h^0(X,mK_X)$ grows as $m^\kappa$ for $m$ large (more precisely, $\kappa(X):=\limsup\limits_{m\to\infty}\frac{\ln h^0(X,mK_X)}{\ln m}$). In general we have $\kappa(X)\in\lbrace-\infty,0,\ldots,\dim X\rbrace$, and a variety is said
to be {\it of general type} if $\kappa=\dim X$.

The Kodaira dimension of a variety is a birational invariant. The minimal model conjecture/program states that any variety of
general type is birational to a {\it canonical model}, i.e.~a
variety with only canonical singularities, and such that on it the
canonical divisor is ample. Thus if the canonical class is ample and
the singularities are canonical, the variety is its own canonical model.
\end{df}
To compute the canonical class of $\A_g$ and $\A_g^*$, one writes
down the explicit volume form $\omega(\tau):=\bigwedge_{i\le
j}\tau_{ij}$ on $\H_g$. To get the class $K_{\A_g}\in\Pic(\A_g)$ one
needs to determine the transformation properties of the form
$\omega$ under the action of $\Sp$. It turns out that
$\omega(\gamma\tau)=\det(C\tau+D)^{-g-1} \omega(\tau)$, which means
that $K_{\A_g}=(g+1)L$. Now determining the class of $K_{\A_g^*}$ is
very easy --- we just need to see how fast $\omega(\tau)$
degenerates as $\tau$ goes to the boundary of $\A_g^*$, i.e.~as say
$\tau_{11}\to i\infty$. Clearly in this case there is one factor in
$\omega$, precisely $d\tau_{11}$, which degenerates, and thus we get
$$
  K_{\A_g^*}=(g+1)L-D.
$$
The same expression is true for $K_{\AP}$.
\begin{dsc}[{\bf The nef cone of $\A_g^*$}]

Determining the nef cone is equivalent to determining the cone of
effective curve classes, as these are dual. From our review of
modular forms we know that $L$ is ample on $\A_g^S$, and thus $L$ is
nef on $\A_g^*$ (where it is a pullback from $\A_g^S$.
Moreover, on the fiber of the map
$\partial\A_g^*=\X_{g-1}\to\A_{g-1}$ over some point
$[B]\in\A_{g-1}$ the restriction $D|_D=-2\T_B$ (see \cite{mumford}),
and thus $-D$ is relatively ample with respect to this contraction map.

There are two easy to construct curve classes in $\A_g^*$. Let
$C_1\subset\A_g^*$ be any curve in the boundary projecting to a point
in $\A_g^S$, i.e.~$C_1\subset B$, where $B$ is the fiber of
$\partial\A_g=\X_{g-1}\to\A_{g-1}$ over $[B]\in\A_{g-1}$. Since $L$
is ample on $\A_g^S$ and $L.C_1=0$, the curve $C_1$ must lie in the
boundary of the cone of effective curves on $\A_g^*$. Dually, $L$ must lie in the
boundary of the nef cone of $\A_g^*$.

Another curve class in $\A_g^*$ one can consider is
$C_2:=\overline{\A_1}\times [B]$, where $[B]\in\A_{g-1}$ is fixed,
i.e.~this is the family of elliptic tails. The intersection
$L.C_2=1/24$ --- this is the (stacky) degree of $L$ on $\A_1$, which can be
computed by computing the appropriate orbifold structure on $\P^1=\H_1/SL(2,\Z)$ or by integrating the volume form over this fundamental domain. The intersection $D.C_2$ is equal to $1/2$ --- there
is exactly one point in the boundary of $\overline{\A_1}$, and the corresponding semiabelian object $\C^*$ has an involution. Thus we
have $(12L-D).C_2=0$. If we had a map of $\A_g^*$ contracting $C_2$,
we would conclude that $12L-D$ is the other boundary of the nef
cone. Unfortunately, such a map is not known, but the result still
holds.
\end{dsc}
\begin{thm}
a) (Hulek and Sankaran, \cite{husa2}) The cone of effective curves
on $\A_g^*$ is generated by $C_1$ and $C_2$, i.e.~the nef cone is
$$
 Nef(\A_g^*)=\lbrace aL-bD\mid a\ge 12 b\ge 0\rbrace,
$$
so the minimal slope of nef divisors is 12.

b) (Hulek \cite{hulek}) For the genus 2 and 3 toroidal
compactifications the same result holds (for $g\le 3$ the perfect cone, central cone, and the second Voronoi compactifications coincide).
\end{thm}
Genus 3 is the highest in which the first and second Voronoi
compactifications coincide. In general $\AP\ne\AV$, and the
birational map from one to the other is regular in neither
direction.

However, in dimension 4 there exists a contracting morphism
$\overline{\A_4}^V\to\overline{\A_4}^P$, with an irreducible
exceptional divisor that we denote by $E$, over the stratum $\A_0\subset\A_4^S$. The explicit geometric
and combinatorial description of the toroidal compactifications in
dimension 4, though very hard, gives an approach to the nef cone of
$\overline{\A_4}^V$ --- here is the result.
\begin{thm}[Hulek and Sankaran \cite{husa2}]\label{pic4}
The nef cone of $\overline{A_4}^P$ is the same as for the partial
compactification, i.e.
$$
 Nef(\overline{A_4}^P)=\lbrace aL-bD\mid a\ge 12 b\ge 0\rbrace.
$$
For the second Voronoi compactification, we have
$$
 \Pic_\Q(\overline{A_4}^V)=\Q L\oplus\Q D\oplus\Q E;
$$
$$
 Nef(\overline{A_4}^V)= \lbrace aL-bD-cE\mid a\ge 12 b\ge 24c\ge 0\rbrace.
$$
\end{thm}
\begin{dsc}[{\bf Canonical model of $\A_g$}]

In \cite{husa1} the question of determining the cone $Nef(\AV)$ for
arbitrary $g$ is posed, but, as explained in \cite{husa2}, it seems
that the dimensions of $\Pic_\Q(\AV)$ grow fast with $g$, and thus this
question, though very interesting especially because of Alexeev's
interpretation of $\AV$ as the functorial compactification,
currently seems beyond reach.

However, $\Pic_\Q(\AP)$ is always two-dimensional, and in view of the
above $g\le 4$ results it is tempting to conjecture that
$Nef(\AP)=Nef(\A_g^*)$. This is indeed the case, as was recently
proven:
\end{dsc}
\begin{thm}[Shepherd-Barron \cite{shepherdbarron}]
In any genus the nef cone of $\AP$ is the same as that of $\A_g^*$,
i.e.~has minimal slope 12:
$$
  Nef(\AP)=\lbrace aL-bD\mid a\ge 12 b\ge0\rbrace.
$$
\end{thm}
Proving this requires a very detailed study of the structure of
$\partial\AP$ and the torus action on it. One describes the strata of
$\AP\to\A_g^S$ over each $\A_i\subset\partial\A_g^S$ explicitly, as
torus fibrations over the fiberwise $(g-i)$'th power of the universal
family $\X_i$ of ppavs over $\A_i$ --- this uses the specific
geometry and combinatorics of the perfect cone decomposition. One
then uses the torus action along the fibers in each stratum to
``average'' any effective curve --- for the perfect cone
compactification we get then a curve on the ``zero-section'' of the torsor, i.e. on $\X_i^{\times(g-i)({\rm fiberwise})}$.

One then uses the fact that the stratum of $\AP$ lying over
$\A_{i-1}\subset\A_g^S$ is up to codimension two essentially the
partial compactification of the power of the universal family over
$\A_i^*$. After more hard work one eventually deduces that if there
exists a curve $C\subset\AP$ projecting to a point of $\A_i$ such
that $(12L-D).C<0$, then there exists such a curve over $\A_{i-1}$,
and then induction yields a contradiction. In doing this, the
explicit understanding of the geometry of the perfect cone enters in
many places and plays a crucial role.
\begin{cor}[Shepherd-Barron \cite{shepherdbarron}]
$\AP$ is the canonical model of $\A_g$ for $g\ge 12$, since
$K_{\AP}=(g+1)D-L$ is then ample.
\end{cor}
The corollary follows from the theorem once it is established that
all the singularities of $\AP$ are terminal, which is done, building
upon the local ideas of the computations from \cite{tai}, in
\cite{shepherdbarron}.

Thus for $g\ge 12$ the minimal model program for $\A_g$ is complete
--- we know that $\AP$ is the canonical model.
\begin{OP}
Determine the canonical model of $\A_g$ for $g<12$.
\end{OP}

\begin{dsc}[\bf Kodaira dimension of $\A_g$]
Comparing the Kodaira dimension of a variety and its compactification
is a bit tricky --- a priori it is not clear that pluricanonical
forms on a variety would extend to a compactification. However, for
$\A_g^*$ there is no problem by the following result.
\end{dsc}
\begin{thm}[Tai \cite{tai}]
Any section of $mK_{\A_g^*}$ extends to a section of $mK_\AP$.
\end{thm}
The study of Kodaira dimension of $\A_g$ was pioneered by Freitag,
who in \cite{freitagpaper} showed that $\A_g^*$ is of general type
for $g$ divisible by $24$, by explicitly constructing many
pluricanonical forms in this case. In \cite{tai} Tai studied the
spaces of modular forms and obtained estimates for the dimension of
the space of pluricanonical forms (see theorem \ref{minslope} and
proof for more details), which allowed him to prove directly from
the definition of Kodaira dimension
\begin{thm}[Tai \cite{tai}]
For $g\ge 9$ the space $\A_g$ is of general type.
\end{thm}
\begin{dsc}[{\bf Effective divisors}]

For any variety $X$ if we have $K_X=E+A$, where $E$ is an effective
divisor, and $A$ is a big $\Q$-divisor\footnote{A divisor $D$ on $X$
is called big if $h^0(X,mD)$ grows as $m^{\dim X}$.}, and the
singularities are canonical, then $X$ is of general type. Since we
know that $L$ is big and nef on $\A_g^*$, it follows that $\A_g^*$,
or, properly speaking, $\AP$ is of general type if we can find an
effective $\Q$-divisor $E$ such that $K_X=E+\e L$, for some $\e>0$,
i.e.~if there exists an effective divisor of slope
$s(E)<s(K_{\A_g^*})=g+1$.

A direct way to construct effective divisors is to consider the zero
loci of explicit modular forms. As observed by Freitag
\cite{freitagbook}, one can consider the modular form
$$
  \t_{\rm null}:=\prod\limits_{\e,\de\in(\frac12\Z/\Z)^g{\rm\
  even}}   \tt\e\de(\tau),
$$
(where even means that the scalar product $4\e\cdot\de=0\mod 2$),
for which the weight and the vanishing order can be easily computed.
This gives the slope $s(\t_{\rm null}) =8+\frac{1}{2^{g-3}}$, which
is less than $g+1$ for $g\ge8$, so this implies that $\A_g$ is of
general type for $g\ge 8$.

Constructing other explicit modular forms of small slope is quite
hard, and if one writes down a random modular form, chances are it
would be of very high slope --- indeed, if a modular form belongs to
a family that has no base locus, then its zero locus must intersect
any curve non-negatively, and thus the modular form defines a nef
divisor, which is thus of slope at least 12.
\end{dsc}

\smallskip
Alternatively one can construct effective divisors on $\A_g$ by
considering loci of abelian varieties satisfying some special
geometric property. This approach has been very successful for moduli
of curves (see \cite{fapo}, \cite{farkas}, \cite{fa22} for recent
results and \cite{farkassurvey} for a survey), but is harder to
pursue for $\A_g$ than for $\M_g$, as there are fewer geometric
constructions known that are associated to a ppav than to an
algebraic curve.

\begin{df}
The {\it Andreotti-Mayer divisor $N_0\subset\A_g^*$} is the closure
in $\A_g^*$ of the locus in $\A_g$ of those ppavs for which the
theta divisor is a singular $(g-1)$-dimensional variety.
\end{df}
Mumford used Grothendieck-Riemann-Roch for universal families and
studied the geometry of the boundary to compute the class of $N_0$.
\begin{thm}[Mumford \cite{mumford}]
The slope of the Andreotti-Mayer divisor is
$s(N_0)=6+\frac{12}{g+1}$; by comparison with $s(K_{\A_g^*})=g+1$ it
follows that $\A_g$ is of general type for $g\ge 7$.
\end{thm}
\begin{rem}
The class of the divisor $N_0$ was later also computed by Yoshikawa
\cite{yoshikawa} by more analytic methods. Since $N_0$ is an
effective geometric divisor in $\A_g^*$, one can ask whether it is
given as the zero locus of a modular form. Work in this direction was
done, and an integral expression for $N_0$ was obtained by Kramer and
Salvati Manni in \cite{krsm}, but there is still more to be
understood about the relationship of the geometry and modular forms
here.
\end{rem}
\begin{OP}
Write down an explicit modular form for which $N_0$ is the zero
locus.
\end{OP}

This of course does not mean that all $\A_g$ are of general type. It
was known classically that $\A_1=\M_1$ and $\A_2=\M_2$ are rational,
and thus of Kodaira dimension $-\infty$.
\begin{thm}\hskip1cm\\
(Katsylo \cite{katsylo}) $\M_3$, and thus also $\A_3$, is
rational.\\
(Clemens \cite{clemens}) $\A_4$ is unirational.\\
(Donagi \cite{donagi}, Mori and Mukai \cite{momu}, Verra
\cite{verra}) $\A_5$ is unirational.
\end{thm}

Thus since the 1980s only the Kodaira dimension of $\A_6$ remained
unknown.

Since $\Pic_\Q(\A_g^*)=\Q^2$, to compute the class of any divisor in
it all that is needed is to compute the intersection numbers of this
divisor with two numerically non-equivalent test curves. Since
$\Pic_\Q(\A_g)=\Q$, only one test curve can be taken to be an
arbitrary curve lying completely in $\A_g$ (these exist for $g\ge 3$,
see \cite{kesa} or section 7 below for a discussion of related
questions). Since $\Pic_\Q(\A_g^S)=\Q$, for the other test curve we
can take a curve in $\A_g^*$ contracted to a point in $\A_g^S$ ---
this means that we can choose a ppav
$[B]\in\A_{g-1}\subset\partial\A_g^S$ general, and take a general
curve $C\in B\subset\partial\A_g^*$.

Then to compute the class of $N_0$ one can do the following: restrict
the universal theta divisor and the universal family
$\Theta_g\subset\X_g$ to a test curve $C$ (and denote the
restrictions $\Theta\subset \X$), and then use the ramification
formula for the map $\Theta\to C$, which would thus give the
intersection number $N_0.C$ in terms of some intersection numbers of
classes $\T$ and $c_1(T_{\X/C})$ on $\X$. Mumford performed this
computation for a test curve $C\subset\A_g$, but over the boundary
relied on the geometric description of $N_0$ to compute the
corresponding coefficient. If one were to try to compute the class of
any other geometrically defined divisor, such a geometric approach
might not work.

However, the intersection theoretic computation can also be carried
out over the boundary. Indeed, in this case $\X$ should be the
universal semiabelian family over a curve $C\subset\partial\A_g^*$
(that is contracted to $[B]\in\A_{g-1}\subset\partial\A_g^S$). This
universal family is in fact the total space of the projectivized
Poincar\'e bundle on $B\times B$ restricted to $B\times C$ --- see,
for example, \cite{alexeev},\cite{huleknotes}. Once the intersection
numbers on this family are computed, the class of $N_0$ (and thus
potentially of other divisors) can be computed directly, without
appealing to the specific geometry of the situation. This was
recently accomplished, and the computation results are as follows.

\begin{prop}[Mumford \cite{mumford}]
For $p:\X\to C$ being the universal family over a test curve
$C\subset \A_g$ the pushforwards are
$$
  p(\T^{g+1})=\frac{(g+1)!}{2}L;\quad p(\T^{g}L)=g!L.
$$
\end{prop}
\begin{prop}[--- and Lehavi \cite{grle}]
For $p:\X\to C$ being the universal semiabelian family over a test
curve $C\subset\p\A_g^*$, contracted to a point
$[(B,\T_B)]\in\A_{g-1}$ the pushforwards are
$$
  p(\T^{g+1})=\frac{(g+1)!}{6}\T_B,\quad p(\T^{g}c_1(T_{\X/C})=0,.
$$
\end{prop}
These results should allow computation of the classes in
$\Pic_\Q(\A_g^*)$ of many geometrically defined divisors ---
unfortunately the ones we have already tried did not give low slope.

\smallskip
The table of slopes of various effective divisors is as follows; here
$N_0^*:=N_0-2\t_{\rm null}$ has slope slightly less than $N_0$, and
was thus used by Mumford:

\smallskip
\begin{center}
\begin{tabular}{|c|r|r|r|}
  \hline
  $\vphantom{\dfrac{1}{2}} g$&$s(K_{\A_g^*})$
  & $s(\t_{\rm null})$&$\ s(N_0^*)\ $\\
  \hline
  4&5&8.5 &8   \\
  5&6&8.25&7.71\\
  6&7&8.13&7.53\\
  7&8&8.06&7.40\\
  \ldots&\ldots&\ldots&\ldots\\
  $\infty$&$\infty$&8&6\\
  \hline
\end{tabular}
\end{center}

\smallskip
Notice that for all genera $g\ge 5$ we in fact have $s(\t_{\rm
null})>s(N_0^*)>6$, and it seems very natural to wonder whether the
minimal slope of $Eff(\A_g^*)$ is always at least 6. This is
absolutely not the case.
\begin{thm}[Riccardo Salvati Manni explained to us how this is obtained
by improving the bounds in Tai \cite{tai}] \label{minslope} There
exists an effective divisor on $\A_g^*$ of slope at most
$$
 \frac{(2\pi)^2}{\sqrt[g]{g!}\,\sqrt[g]{2\zeta(2g)}}.
$$
\end{thm}
\begin{cor}
The slope of the effective cone goes to zero as $g$ increases:
$\lim\limits_{g\to\infty} s(Eff(\A_g^*))=0$.
\end{cor}
\begin{proof}[Proof of the corollary]
Indeed, we have $\lim\limits_{g\to\infty} \zeta(2g)=1$, and
$\sqrt[g]{g!}\sim g/e$, so for large $g$ the asymptotics of the
expression in the theorem is $\frac{(2\pi)^2e}{g}$, which tends to 0
as $g$ increases.
\end{proof}
\begin{proof}[Proof of the theorem]
The improvement of Tai's result is obtained by looking more
carefully at his dimension estimates. For convenience, we recall
Tai's notations and results.

Denote by $A_{g,k}$ the vector space of scalar modular forms on
$\A_g$ of weight $k(g+1)$. The reason for this notation is that
$A_{g,k}$ are forms in $kK_{\A_g}$, i.e.~$k$-pluricanonical forms.
Tai computes the asymptotics of the dimension for $g$ fixed and $k$
large (\cite{tai}, Proposition 2.1):
$$
  \dim A_{g,k}\sim 2^{\frac{(g-1)(g-2)}{2}}[k(g+1)]^{\frac{g(g+1)}{2}}
  \prod\limits_{j=1}^g \frac{(j-1)!}{(2j)!}B_j,
$$
where $B_j$ are the even Bernoulli numbers.

The slope of a modular form is its weight divided by the vanishing
order at the boundary. Thus Tai defines (page 429)
$\T_{g-1,m}^{k}(\ell)$ to be essentially the space of all possible
expansions of {\it weight $k(g+1)$} modular forms on $\A_g(\ell)$
near the boundary $\partial\A_g^*(\ell)$, vanishing to order $m$
along the boundary (this is somewhat confusing in \cite{tai} --- he
does not have the upper index $k$ in notations, which is important
for the computation). Such a boundary expansion determines the
modular form uniquely; more precisely $\dim
\T_{g-1,m}^{k}(\ell)^{\rm even}=\dim
H^0(\A_g^*(\ell),k(g+1)L-mD)$, where ``even'' means that we are
taking the even expansions, which are roughly one half of all
expansions.

Thus if for some $M$ we have $\dim A_{g,k}>\sum\limits_{m\le M}\dim
\T_{g-1,m}^{k}(1)$, it follows that there must exist a form in
$A_{g,k}$ with boundary vanishing order at least $M$ and thus slope
at most $\frac{k(g+1)}{M}$. One then estimates (\cite{tai}, Corollary 2.6)
$$
 \dim \T_{g-1,m}^{k}(1)\sim (2m)^{g-1}\dim A_{g-1,k}.
$$
Combining this with the formula for $\dim A_{g,k}$ and taking the
sum, we get (this is the last formula on page 431 of \cite{tai} ---
be warned that there $M=k$, and one needs to carefully retrace Tai's
computations to verify that on the right-hand-side one of the two
places $k$ appears it should now be $M$, while in the other it is
still $k$):
$$
 \sum\limits_{m\le M}\dim\T_{g-1,m}^k(1)^{\rm even}\sim
 2^{\frac{(g-2)(g-1)}{2}}\frac{M^g}{g}[k(g+1)]^{\frac{g(g-1)}{2}}
  \prod\limits_{j=1}^{g-1}\frac{(j-1)!}{(2j)!}B_j,
$$
for $k$ and $M$ large enough.

Finally, to show the existence of a modular form of slope $s$, we
need to have a modular form of weight $N:=k(g+1)$ (for $k$ very
large), vanishing at the boundary to order $M:=N/s$. Such a modular
form must exist if
$$
 1<\frac{\dim A_{\frac{N}{g+1}}}{\sum\limits_{m\le M}
 \dim\T_{g-1,m}^k(1)^{\rm even}}\sim
 \frac{2^{\frac{(g-1)(g-2)}{2}}N^{\frac{g(g+1)}{2}}
  \prod\limits_{j=1}^g \frac{(j-1)!}{(2j)!}B_j}
 {\frac{1}{g}2^{\frac{(g-2)(g-1)}{2}}\left(\frac Ns\right)^gN^{\frac{g(g-1)}{2}}
  \prod\limits_{j=1}^{g-1}\frac{(j-1)!}{(2j)!}B_j}
$$
$$
 =s^gB_g\frac{g!}{(2g)!}=s^g\frac{g!(2g)!2\zeta(2g)}{(2g)!(2\pi)^{2g}}=
 s^g\frac{g!2\zeta(2g)}{(2\pi)^{2g}},
$$
where we used the explicit formula for $B_g$ in terms of the zeta
function.

This inequality holds for
$$
  s>\frac{(2\pi)^2}{\sqrt[g]{g!}\sqrt[g]{2\zeta(2g)}},
$$
and thus there exist modular forms of this slope.
\end{proof}

\smallskip
Comparing the slope bound from theorem \ref{minslope} to $s(N_0^*)$,
we see that at least for $g\ge 13$ (and likely for smaller $g$ as
well) $N_0^*$ cannot be the effective divisor of the smallest slope. This
leaves the following important question wide open.
\begin{OP}
What is the slope of the cone $Eff(\A_g^*)$?
\end{OP}
Since the slope of $Eff(\overline{\M_4})=8.4$ is known, and
$\M_4\subset\A_4$ is codimension one, given by the Schottky modular
form of slope 8, it follows that $Eff(\A_4^*)=Eff(\overline{\A_4}^P)=\lbrace aL-bD \mid a\ge
8b\ge 0\rbrace$ has slope 8. However, already $Eff(\A_5^*)$ is not
known. Oura, Poor and Yuen \cite{poor} have been studying this
question from the point of view of code polynomials etc., but a
complete answer still seems beyond reach.
\begin{rem}
It is very interesting to compare what we know about the slopes of
the effective cones of $\overline{\M_g}$ and $\A_g^*$. A
long-standing slope conjecture for $\M_g$ predicted that the
Brill-Noether divisor had minimal slope, and these slopes tended to
$6$ as $g$ went to infinity. The slope conjecture was disproven by
Farkas and Popa \cite{fapo}, and divisors of smaller slopes have been
constructed by them and by Farkas \cite{farkas}, \cite{fa22}.
However, all of these have slope at least 6, while it is not even
clear whether there exists a genus-independent lower bound for slopes
of effective divisors on $\overline{\M_g}$.

By the above theorem, there is no such bound for $\A_g^*$, and it is
tempting to try to apply techniques similar to Tai's to
$\M_g\subset\A_g$ to prove that there is no such bound for $\M_g$,
either. Since the dimension count in theorem \ref{minslope} produces
effective divisors on $\A_g$ of slope smaller than
$6.5=s(K_{\overline{\M_g}})$ for $g\ge 14$, and since in this range
(except for $\M_{14}$, which is known to be unirational \cite{ve2})
the Kodaira dimension $\kappa(\overline{\M_g})$ is not always known,
it would also be very interesting to try to use Tai's
dimension-counting techniques to approach this computation, but this
also seems hard.
\end{rem}
\begin{rem}
There is also a very curious coincidence: the slope of the
Brill-Noether divisor on $\overline{\M_g}$ is equal to
$6+\frac{12}{g+1}$, the same as the slope of $N_0$ on $\A_g^*$. For
$g\ge 4$ under the Torelli map we have $\M_g\subset N_0$, but since
$\M_g\subset\A_g$ is of high codimension for $g$ large, so far this
equality of slopes seems to be just a numerical coincidence, though
a very strange one. Finding a reason for it, if there is one, could
shed more light on the relationship of $Eff(\overline{\M_g})$ and
$Eff(\A_g^*)$, and perhaps on the geometry of the Schottky problem.
\end{rem}

\section{Homology and Chow rings: intersection theory on $\A_g$}
Having discussed the birational geometry, i.e.~divisors, in the
previous section, we now review the progress made in understanding
the higher-dimensional cohomology and Chow rings of $\A_g$ and
compactifications, and the intersection theory.

\begin{df}
In $\Pic_\Q(\A_g)$ we had one natural class --- the Hodge line bundle
$L=\det \E$. Similarly, the most natural homology or Chow classes on
$\A_g$ are the {\it Hodge classes}, i.e.~the Chern classes of the
Hodge bundle
$$
 \ll_i:=c_i(\E).
$$
\end{df}
The cohomology of the open space $\A_g$ is the same as the group
cohomology of $\Sp$, and a lot is known about it. Notice that
choosing $[A]\in\A_h$ gives a natural embedding
$\A_g\hookrightarrow\A_{g+h}$, by taking the Cartesian product with
$A$. All of these embeddings are homotopic, and thus one can talk of
the stable cohomology of $\A_g$. In comparison, for $\M_g$ there is
no natural map $\M_g\hookrightarrow\M_{g+h}$, as taking the product
with a fixed curve of genus $h$ gives reducible stable curves, which
lie in $\partial\overline{\M_{g+h}}$ --- so for $\A_g$ we get an analog of
Harer's stability for free. The stable cohomology of $\A_g$, the
same as that of $\Sp$, has been computed much before the recent proof
by Madsen and Weiss \cite{mawe} of Mumford's conjecture on the
stable cohomology of $\M_g$.
\begin{thm}[Borel, see \cite{knudson} for an exposition]
The stable cohomology ring of $\A_g$ is freely generated by a class
in each dimension $4k+2$, i.e.~for any fixed $n$ there exists a
$G(n)$ (some explicit formula for $G(n)$ is actually known) such
that for all $g>G(n)$ the cohomology ring $H^*_\Q(\A_g)$ in
dimensions $\le n$ is the free algebra generated by the odd Hodge
classes $\ll_1,\ll_3,\ll_5,\ldots$.
\end{thm}
In comparison, the stable cohomology of $\M_g$ is generated by a
class in every even dimension, i.e.~while the classes $\ll_{2k}$ on
$\A_g$ are expressible algebraically in terms of $\ll_{2k+1}$ (and this relation of course also holds over $\M_g$), on $\M_g$ there are also stably algebraically independent from $\ll$'s Miller-Morita-Mumford's classes $\kappa_{2k}$.

Moreover, for $\A_g$ there exist also product maps for
compactifications: $\A_g^S\times\A_h^S \to\A_{g+h}^S$,
$\overline{\A_g}^P\times\overline{\A_h}^P \to\overline{\A_{g+h}}^P$,
and $\overline{\A_g}^V\times\overline{\A_h}^V
\to\overline{\A_{g+h}}^V$. Thus we are naturally led to ask whether
the stable homology can be computed for various compactifications of
$\A_g$. The answer is in fact known for the Satake compactification.
\begin{thm}[Charney and Lee \cite{chle}]
The stable homology ring of $\A_g^S$ is freely generated by the odd
Hodge classes $\ll_{2k+1}$, for $k\ge 0$, and some other classes
$\alpha_{2k+1}$, for $k\ge 1$.
\end{thm}
It appears that the classes $\alpha$ may not be algebraic, but the algebraic geometry interpretation of this result is still now know. The stable homology of toroidal compactifications is completely unknown.
\begin{OP} What are the
stable homology rings, or maybe Chow rings, if this makes sense, of
$\AP$ and $\AV$?
\end{OP}
We thank Nicholas Shepherd-Barron for discussions relating to this
question, drawing our attention to \cite{chle}, and telling us about
the following considerations.

These cohomology rings could be understood as the cohomology rings of
the corresponding inductive limits $\overline{\A_\infty}^P$ and
$\overline{\A_\infty}^V$ --- these actually exist in the appropriate
monoid category, but their topology may depend on the choice of the
base point for the embedding $\A_g\hookrightarrow\A_{g+1}$. Moreover,
the cohomology of ind-limits is naturally a graded Hopf algebra, and
thus by a theorem of Milnor and Moore \cite{mimo} is a product of a
polynomial ring and an exterior algebra. We note that the number of
irreducible boundary components of $\AV$ grows unboundedly as $g$
grows, as thus the stable homology of $\AV$ may only exist in some
sense in which similarly the stable homology of the Deligne-Mumford
compactification $\overline{\M_g}$ could exist.

The topic of stable modular forms (i.e.~the structure of the limit
$\A_\infty^S$) has been studied at least since the work of Freitag
\cite{freitagstable}. The space $\overline{\A_\infty}^P$ is of
interest in particular due to the work of Shepherd-Barron: one can
try to think of it as the universal canonical model for $\A_g$ in
some sense.

We also note that there do not exist any ``stable compactly
supported'' cohomology classes, i.e.~there cannot exist families of
complete subvarieties of $\A_g$ of the same codimension for all $g$
--- by theorem \ref{oort} below the codimension of a complete
subvariety of $\A_g$ must be more than $g$.

\begin{dsc}[{\bf Tautological ring}]

Analogously to the case of $\M_g$ (see \cite{tommasi} for the case of
$\M_4$ and \cite{grpa} for the case of $\M_{1,11}$), there may exist cohomology classes in $\A_g$ not lying in the algebra
generated by the Hodge classes. There has been much progress for
$\overline{\M_g}$ in studying the {\it tautological ring} --- the
subring of the Chow generated by the naturally defined classes; one
major goal being proving Faber's conjecture \cite{fabertaut}. The
tautological ring can also be studied for $\A_g$ --- one simply
considers the subring of the Chow generated by the Hodge classes
$\ll_i$. This has been determined entirely.
\end{dsc}
\begin{thm}[van der Geer \cite{vdgeer1} for $\A_g$, Esnault and Viehweg
\cite{esvi} for a compactification ] For an appropriate toroidal
compactification the tautological subring of
$CH_\Q^*(\overline{\A_g})$ generated by the Hodge classes has only
one relation:
\begin{equation}\label{chowrel}
 (1+\ll_1+\ll_2+\ldots+\ll_g)(1-\ll_1+\ll_2-\ldots+(-1)^g\ll_g)=1.
\end{equation}
The tautological subring of $CH_\Q^*(\A_g)$ has one more relation:
$\ll_g=0$.
\end{thm}
Writing out all the terms of relation (\ref{chowrel}), we can
immediately see that the even Hodge classes are expressible in terms
of the odd Hodge classes. For example equating to zero the $CH^2$
term gives $2\ll_2=\ll_1^2$, the $CH^4$ term gives
$\ll_4=2\ll_1\ll_3-\ll_2^2=2\ll_1\ll_3-\frac14 \ll_1^4$, etc.

Note that the above equalities are in $CH_\Q^*(\A_g)$, and thus one
can wonder what happens in $CH_\Z^*$. The torsion of $\ll_g\in
CH_\Z^*(\A_g)$, and subvarieties representing it on the
compactification (since $\ll_g$ is zero on $\A_g$, it defines some
subvariety of the boundary) were studied by Ekedahl and van der Geer
\cite{ekvdg1}, \cite{ekvdg2}. It is interesting to compare this to
the recent work of Galatius, Madsen and Tillmann \cite{tillmann} on
the divisibility of the tautological classes on $\M_g$.

The full homology and Chow rings (as opposed to just the
tautological subring) were computed for $\A_g$ for $g\le 3$. The
results for genera 1 and 2 are classical, and the same as for the
moduli space of curves. For genus 3 we have the following two
computations.
\begin{thm}[Hain \cite{hain}]
The dimensions of the rational cohomology groups for $\A_3$ and its
Satake compactification $\A_3^S$ are
\begin{center}
\begin{tabular}{|c|rrrrrrr|}
\hline
n\ &\ 0&\ 2&\ 4&\ 6&\ 8&10&12\\
\hline
$\vphantom{\vbox{\vskip13pt}}$ $\dim H^n_\Q(\A_3)$&1&1&1&2&0&0&0\\
$\vphantom{\vbox{\vskip13pt}}$ $\dim H^n_\Q(\A_3^S))$&1&1&1&2&1&1&1\\
\hline
\end{tabular}\ ,
\end{center}
while the homology in all odd dimensions is zero. Moreover, the space $H^6(\A_3)$ is described explicitly as a mixed Hodge structure.
\end{thm}

\begin{thm}[van der Geer \cite{vdgeer2}]
The Chow groups of $\overline{\A_3}$ (which are actually equal to the cohomology, though this is not a priori clear) have the following dimensions
\begin{center}
\begin{tabular}{|c|rrrrrrr|} \hline
n&\ 0&\ 1&\ 2&\ 3&\ 4&\ 5&\ 6\\
\hline
$\vphantom{\vbox{\vskip13pt}}$ $\dim CH^n_\Q(\overline{\A_3})$&1&2&4&6&4&2&1\\
\hline
\end{tabular}\ .
\end{center}
\end{thm}

\smallskip
In fact van der Geer describes the generators of all the Chow groups
and the entire ring structure. While it seems very hard to describe
the entire Chow ring in higher genera, one result that could
potentially be generalized is the intersection theory of divisors,
as we know that $\Pic_\Q(\AP)$ is always two-dimensional. For genus
3 the numbers are

\begin{thm}[van der Geer \cite{vdgeer2}]
The intersection numbers of divisors on $\overline{\A_3}$ are
\end{thm}

\smallskip
\begin{center}
\begin{tabular}{|c|c|c|c|c|c|c|} \hline
$\vphantom{\dfrac{1}{2}} L^6$& $L^5D$&$L^4D^2$&
$L^3D^3$ &$L^2D^4$&$LD^5$&$D^6$\\
\hline $\vphantom{\dfrac{1}{2}}
\frac{1}{181440}$&$0$&$0$&$\frac{1}{720}$&$0$
&$-\frac{203}{240}$&$-\frac{4103}{144}$\\
\hline
\end{tabular}\ .
\end{center}

\smallskip
Compared to $\AP$, the intersection theory on $\overline{\M_g}$ has
been extensively studied. Using Faber's intersection computations
program \cite{faber} for $\overline{\M_4}$ and the computation
of the class of $\overline{\M_4}$ as a divisor in
$\overline{\A_4}^P$ and $\overline{\A_4}^V$ by Harris and Hulek
\cite{hahu}, in a recent work we have determined the intersection
theory of divisors in genus 4.
\begin{thm}[Erdenberger, ---, Hulek \cite{egh}]
a) The intersection numbers of divisors on $\overline{\A_4}^P$
(recall $\Pic_\Q(\AP)=\Q L\oplus\Q D$) are {\rm

\smallskip
\begin{center}
\begin{tabular}{|c|c|c|c|c|}
\hline $\vphantom{\dfrac{1}{2}} L^{10}$&$L^6D^4$&$L^3D^7$&$LD^9$&$D^{10}$\\
\hline $\vphantom{\dfrac{1}{2}}
\frac{1}{907200}$&$-\frac{1}{3780}$&$-\frac{1759}{1680}$&$\frac{1636249}{1080}$&
$\frac{101449217}{1440}$\\
\hline
\end{tabular}\ ,
\end{center}

\smallskip} while all others are zero.

\smallskip
b) For $\overline{\A_4}^V$, recall from theorem \ref{pic4} that there
is a contracting morphism $\pi:\overline{\A_4}^V
\to\overline{\A_4}^P$, with exceptional divisor $E$, and
$\Pic_\Q(\overline{\A_4}^V)=\Q L\oplus\Q D\oplus \Q E$. For our
purposes use $F:=D+4E$ instead of $D$ in the basis: $F\subset\overline{\A_4}^V$ is the
pullback of $D\subset\overline{\A_4}^P$ under $\pi$. Then we have
$$
 \langle L^i F^j\rangle_{\overline{\A_4}^V}=\langle L^i D^j\rangle_{
 \overline{\A_4}^P};\quad \langle
 E^{10}\rangle_{\overline{\A_4}^V}=-\frac{35}{24},
$$
while all other intersection numbers $\langle L^i
F^jE^k\rangle_{\overline{\A_4}^V}$ with $k(i+j)\ne 0$ are zero.
\end{thm}
We now remark that in the above results many of the intersection
numbers turn out to be zero, and thus the following conjecture is
plausible
\begin{conj}[Erdenberger, ---, Hulek \cite{egh2}]
An intersection number $\langle L^i D^{\frac{g(g+1)}{2}-i}
\rangle_\AP$ is zero unless $i=\frac{k(k+1)}{2}=\dim\A_k$ for some
$k\le g$.
\end{conj}
This is indeed true by inspection of the above numbers for $g\le 4$,
and  by explicitly studying the geometry of the boundary strata of
$\AP$ and the intersection numbers on them, the following result was
also obtained.
\begin{thm}[Erdenberger, ---, Hulek \cite{egh2}]
The above conjecture is true for $i>\frac{(g-3)(g-2)}{2}$; explicit
formulae for the non-zero intersection numbers in this range are also
obtained.
\end{thm}
The above considerations suggest that the homology and intersection
homology of $\AP$ and $\A_g^S$ could be related; it is natural to
look more generally at the full Chow and cohomology rings instead of
just the top intersection numbers of divisors. Since the class of
$\overline{\M_4}\subset\overline{\A_4}^V$ is known and much is known
about the Chow ring and cohomology of $\M_4$, there is the following
natural
\begin{OP}
Determine the cohomology and Chow rings for $\overline{\A_4}^P$,
$\overline{\A_4}^V$, or at least for $\A_4$.
\end{OP}
Tommasi \cite{tommasi} recently computed the cohomology of $\M_4$,
which turns out to have an odd class.
\begin{OP}
Is some odd cohomology $H^{2k+1}(\A_g)$ ever non-zero? In
particular, do $\A_4$ or its compactifications have any odd
cohomology?
\end{OP}
\section{Special loci: subvarieties of $\A_g$}
In section \ref{biratgeom} we discussed the question of
constructing geometric divisors on $\A_g^*$. In the previous section
we discussed the Chow and homology rings of $\A_g$ and its
compactifications. We will now consider the question of constructing
and studying subvarieties of $\A_g$ of any dimension. One possible
motivation for researching this would be to try to see if perhaps
the cohomology is supported on a closed subvariety. On the other
hand, stratifying $\A_g$ in a geometrically meaningful way could
shed more light on the geometry of individual abelian varieties,
depending on which stratum they lie in, and yield results related to
characterizing geometrically constructible loci. Many of the
constructions and problems we survey are discussed in more detail in
\cite{bila}.

\begin{dsc}[{\bf Complete subvarieties}]

In \cite{vdgeer1} van der Geer showed that $\ll_1^{\frac{g(g-1)}{2} +1}=0 \in
CH^*_\Q(\A_g)$. Since $\ll_1$ is ample on $\A_g$, it follows that
there cannot exist a closed subvariety $\A_g$ of dimension larger
than $\frac{g(g-1)}{2}$ (i.e.~of codimension less than $g$), since
otherwise the top power of $\ll_1$ on it would have to be non-zero, contradicting the above equality. However, it is known that
$\ll_1^{\frac{g(g-1)}{2}}\ne 0\in CH^*_\Q(\A_g)$, so it natural to ask if
there exists a codimension $g$ closed subvariety
$X\subset\A_g$, which could then perhaps carry all the cohomology
(i.e.~such that $H^*(X)=H^*(\A_g)$)? We discuss in section
\ref{charp} that in characteristic $p$ there exists a complete codimension $g$ subvariety of $\A_g$, but over $\C$
this was conjectured by Oort (stated in \cite{vdgoo}) not to be the
case. This was recently proven:
\end{dsc}
\begin{thm}[Keel and Sadun \cite{kesa}]\label{oort}
Over $\C$, there does not exist a complete subvariety of $\A_g$ of
codimension $g$.
\end{thm}
This leads to the following
\begin{OP}
What is (over $\C$) the maximal dimension of a complete subvariety
of $\A_g$?
\end{OP}
Since $\partial\A_g^S$ is codimension $g$, if we start intersecting
general hypersurfaces in $\A_g^S$, then once the dimension of the
intersection drops down to $g-1$, we know that it generally should
not intersect the boundary --- thus there exist complete
subvarieties of $\A_g$ of dimension $g-1$. The theorem above says
that the maximal dimension of a complete subvariety of $\A_g$ cannot be greater than
$\frac{g(g-1)}{2}-1$. We do not have any reasons to believe that
either the lower or upper bound are close to the actual maximal
dimension of subvarieties. Instead of studying the maximal dimension
of a closed subvariety of $\A_g$, one can also ask for the maximal
dimension of a closed subvariety of $\A_g$ passing through a general
point, etc. --- some questions in this direction, for both $\M_g$
and $\A_g$, are discussed by Izadi in \cite{izadi}.

One can also consider the following related
\begin{OP}
What is the cohomological dimension of $\A_g$, i.e.~what is the
smallest $n$ such that for any coherent sheaf $\F$ on $\A_g$ we have
$H^k(\A_g,\F)=0\ \forall k>n$?
\end{OP}
It is clear that if the cohomological dimension is $n$, then the
maximal possible dimension of a complete subvariety is at most $n$,
but we are not aware of a bound going the other way. For $\M_g$
it is conjectured by Looijenga that the cohomological dimension is
equal to $g-2$, and in fact that $\M_g$ can be covered by $g-1$
affine open sets, while for $\A_g$ we do not even have a conjecture.
The issue of cohomological dimension and affine covers was recently
studied by Roth and Vakil \cite{rova}.

\begin{dsc}[{\bf Stratifications of $\A_g$}]

As we saw above, constructing (over $\C$) explicit complete
subvarieties of $\A_g$ is very hard. Maybe it is easier to construct
some non-complete subvarieties? One can consider the loci of ppavs
given by various geometric constructions: Jacobians, Pryms,
intermediate Jacobians, etc., but all of these seem to be, for $g$
large enough, of exceedingly high codimension in $\A_g$, and thus
probably do not capture much of the geometry of $\A_g$. Thus it is
natural to wonder whether one can define stratifications of $\A_g$
and obtain some geometric information about each of the strata.
\end{dsc}
\begin{df}
We define the {\it Andreotti-Mayer locus} $N_k\subset\A_g$ to be the
locus of ppavs for which $\dim\Sing
\T\ge k$. Clearly we then have
$$
 \emptyset=N_{g-1}\subseteq N_{g-2}\subseteq\ldots
 \subseteq N_1\subsetneq N_0\subsetneq N_{-1}=\A_g.
$$
(In \cite{mumford} Mumford proved $N_1\subsetneq N_0$.)
\end{df}
These loci were originally introduced as an approach to the Schottky
problem:
\begin{thm}[Andreotti and Mayer \cite{anma}]
$N_{g-4}$ contains the Jacobian locus as an irreducible component;
$N_{g-3}$ contains the hyperelliptic locus.
\end{thm}
The locus $N_{g-4}$ in low genera was studied by Beauville
\cite{beauville} and Debarre \cite{de3}, \cite{de2} who described
the extra components in it, other than the Jacobian locus,
explicitly. One can also ask what are the dimensions of other
Andreotti-Mayer loci.
\begin{thm}[Ciliberto and van der Geer \cite{civdg},\cite{amsp}]
For all $k\le g-3$ we have $\operatorname{codim}N_k\ge k+2$ (for
$k\ge g/3$ this bound can be improved to $k+3$).
\end{thm}
Is this a reasonable bound for codimension? The codimension in
$\A_g$ of the Jacobian locus, which is a component of $N_{g-4}$, is $\frac{(g-3)(g-2)}{2}$. A na\"\i ve, and thus completely
unjustified, dimension count for the number of conditions for a
point to be in $N_k$ seems to indicate that the codimension should
indeed be quadratic in $k$. This motivates the following:
\begin{conj}[Ciliberto and van der Geer \cite{civdg}]\label{conjcodim}
Within the locus of simple\footnote{We remind that a ppav is called
simple if it does not have an abelian subvariety. A very general
ppav is simple.} abelian varieties, ${\rm codim}\,
N_k\ge\frac{(k+1)(k+2)}{2}$.
\end{conj}
Notice that this conjectural bound is exact for the Jacobian locus
and for the hyperelliptic locus. This conjecture, however, seems
very hard, as even the answer to the following question is unknown
\begin{OP}[Ciliberto and van der Geer \cite{civdg}]
Is it possible that there exists some $k<g-4$ such that
$N_k=N_{k+1}$?
\end{OP}

We know that $N_{g-3}$ contains the hyperelliptic locus. What can we
say about $N_{g-2}$? Consider a decomposable\footnote{We remind that
a ppav is called decomposable if it is isomorphic (with
polarization) to a product of two lower-dimensional ppavs. The term
``reducible'' is often used instead of ``decomposable''.} ppav
$A=A_1\times A_2$. We then have
$$
 \T_A=\left(A_1\times\T_{A_2}\right)\,\cup\,\left(
 \T_{A_1}\times A_2\right)
$$
and thus
$$
 \Sing\T_A\,\supset\,\T_{A_1}\times\T_{A_2},
$$
so for decomposable ppavs $\dim\Sing\T=g-2$, i.e.~$\bigcup\limits_i
\A_i\times\A_{g-i}\subset N_{g-2}$. Since the codimension in $\A_g$
of the locus of decomposable abelian varieties is only $g-1$, the
condition of abelian variety being simple was needed in conjecture
\ref{conjcodim}. Arbarello and De Concini conjecture in \cite{ardc}
that $N_{g-2}$ is in fact equal to the locus of decomposable abelian
varieties. This was proven to be true.
\begin{thm}[Ein and Lazarsfeld \cite{eila}]
$N_{g-2}$ is equal to the locus of decomposable abelian varieties.
\end{thm}
The loci $N_k$ are of great interest, but very hard to study, as
even their dimensions are still not known. From the analytical point
of view it is very hard to determine the dimension of a solution set
of a certain system of equations (singular points are where
$\t(\tau,z)={\rm grad}_z\t(\tau,z)=0$). Thus one wonders if it could
be easier to look at some local singularity conditions instead.
\begin{df}
We denote by $\Sing_k\T:=\lbrace x\in A \mid {\rm mult}_x\T\ge
k\rbrace $ the {\it multiplicity $k$ locus} of the theta divisor,
i.e.~the locus of points $z$ where the theta function, as a function
of $z$, has multiplicity at least $k$, i.e.~such that the theta
function and its partial $z$-derivatives up to order $k-1$ vanish.
By the heat equation this means that all partial $\tau$-derivatives
of the theta function up to order $\lfloor \frac {k-1}{2}\rfloor$
vanish at $(\tau,z)$.
\end{df}
Since multiplicity is a local condition, it is natural to study it
from the point of view of singularity theory and multiplier ideals.
This was done quite successfully.
\begin{thm}[Koll\'ar \cite{kollar}]\label{thmkol}
The pair $(A,\T)$ is log canonical; thus ${\rm codim}_A(\Sing_k\T)\ge
k$. In particular the multiplicity of the theta function at any
point is at most  $g$.
\end{thm}
\begin{OP}
Give a direct analytic proof of this theorem, or at least of the
fact that the theta function cannot vanish at any point to order higher than $g$.
\end{OP}
Though the statement is entirely elementary, we have no idea on how
to approach this problem.
\begin{df}
We define the {\it multiplicity locus} $S_k\subset \A_g$ to be the
locus of abelian varieties for which $\Sing_k\T$ is non-empty. We
then have
$$
 \emptyset=S_{g+1}\subsetneq S_g\subseteq\ldots\subseteq S_2=N_0
 \subsetneq S_1=\A_g.
$$
\end{df}
Similarly to the discussion above for $N_{g-2}$, one can see that
for a $k$-fold product of abelian varieties we have
$\Sing_k\ne\emptyset$, thus in particular products of $g$ elliptic
curves lie in $S_g$.
\begin{thm}[Smith and Varley \cite{smva}]
$S_g=\lbrace$products of $g$ elliptic curves$\rbrace$.
\end{thm}
This is a special case of a more general theorem
\begin{thm}[Ein and Lazarsfeld \cite{eila}]
If for some $k>1$ we have ${\rm codim}_A(\Sing_k\T)=k$, then $A$ is
decomposable.
\end{thm}
This result allows one to say something about ppavs for which
$\Sing_k\T$ has the maximal possible dimension. What happens if the
dimension is one less --- are these ppavs special in any way?
Another question is
\begin{OP}
What is the maximal $k$ for which $S_k$ contains indecomposable
abelian varieties?
\end{OP}
One can also try to ask the same question for sections of multiples
of the theta bundle on an abelian variety, rather than only for the
theta function. This has been investigated by Hacon \cite{hacon},
and Debarre and Hacon \cite{deha}, with results generalizing theorem
\ref{thmkol}. However, we note that by Riemann's theta singularity
theorem for Jacobians, and by its generalizations for Prym varieties
--- see, for example, \cite{casalaina}, the maximal multiplicity of
the theta function for Jacobians and Pryms is
$\lfloor\frac{g+1}{2}\rfloor$. Since the dimension of $\Sing\T$ for
hyperelliptic Jacobians is largest possible for indecomposable
ppavs, and Pryms lie in $N_{g-6}$ (see \cite{de1}) it is natural to
make the following
\begin{conj}
The maximal multiplicity of the theta function for indecomposable
ppavs is equal to $\lfloor\frac{g+1}{2}\rfloor$, i.e.~$S_{\lfloor\frac{g+3}{2}\rfloor}$ is a subvariety of the locus of
decomposable abelian varieties.
\end{conj}
We do not know of an approach to this conjecture short of trying to
define a Prym-like construction for arbitrary ppavs, which would be
very hard, and likely not possible. Another obvious question
to ask is
\begin{OP}
What is the dimension of $S_k$? Is it possible to have
$S_k=S_{k+1}$?
\end{OP}
These are also entirely open. Some attempts to study these conditions by degeneration techniques were made in \cite{amsp}, \cite{grsm4}.

\begin{dsc}[{\bf Seshadri constants}]

The above stratifications of $\A_g$ encode some geometric
information about the theta divisor. The multiplicity is a local
invariant of the theta divisor, but from the point of view of the
modern study of singularities, the multiplicity may not be the best
invariant. Something perhaps more intrinsic is the following.
\end{dsc}
\begin{df}
Given a variety $X$ with a divisor $D$ the {\it Seshadri constant}
is defined to be
$$
 \e(X,D):=\inf\limits_{x\in C\subset X} \frac{C.D}{{\rm
 mult}_x(C)},
$$
where the infinum is taken over all points $x\in X$, and all curves
$C\subset X$ passing through the point $x$.

This is a very important invariant of a pair $(X,D)$ --- for example
the Seshadri constant is positive if and only if $D$ is ample.
\end{df}
One can study the Seshadri constants of general ppavs and then of
special loci in $\A_g$, and see whether the Seshadri constants in
fact capture some geometric information.

\begin{thm}[Lazarsfeld \cite{lazarsfeld}]
There exists a constant $c$ independent of $g$ such that for the Jacobian $(J,\T)$ of any curve of genus $g$ the Seshadri constant $\e(J,\T)\le c\sqrt{g}$.
\end{thm}
In comparison, for general ppavs we have
\begin{thm}[Lazarsfeld \cite{lazarsfeld}, see also Bauer \cite{bauer}]
For a general ppav the Seshadri constant is at least of the order of
a constant times $\sqrt[g]{g!}$.
\end{thm}
There is also the following conjecture.
\begin{conj}[Debarre \cite{debarreFANO}, following Lazarsfeld]
For $g\ge 4$, if $\e(A,\T)<2$, then either $A$ is decomposable, or
it is a hyperelliptic Jacobian.
\end{conj}
\begin{rem}
It appears that recent results of Krichever \cite{krichever},
\cite{trisecant} provide techniques that could potentially be applied
in an attempt to prove the so-called $\Gamma_{00}$ conjecture of van
Geemen and van der Geer \cite{vgvdg}, which is closely related to the
half-degenerate case of the trisecant conjecture. As pointed out in
\cite{debarreFANO}, the $\Gamma_{00}$ conjecture would imply this
characterization of hyperelliptic Jacobians.
\end{rem}

This leads one to hope that perhaps a characterization of Jacobians
by Seshadri constants could be possible, or that one could better
understand the stratification of $\A_g$ by the value of the Seshadri
constant. However, this is not so simple:
\begin{thm}[Debarre \cite{debarreFANO}, see also Lazarsfeld \cite{lazarsfeld} for Jacobians]
There exist Jacobians with Seshadri constants at least constant
times $\ln g$. However, in each genus $g\ge 4$ there exist ppavs
that are not Jacobians, but with Seshadri constant equal to 2.
\end{thm}
\begin{OP}
What is the actual order of growth of the Seshadri constants for
generic Jacobians? We know it is between $\ln g$ and $\sqrt{g}$, but
it seems not much more is known.
\end{OP}
Thus the stratification by the value of the Seshadri constant is
also quite complicated. We believe that, if possible, giving a
meaningful answer to the following loosely-phrased question would be
extremely useful in understanding the geometry of $\A_g$.
\begin{OP}
Define a stratification of $\A_g$ with geometrically tractable strata,
i.e.~such that the number of the strata, and at least their dimensions are computable. Try to also say
something about the special properties of the geometry of ppavs in
each strata, perhaps inductively in the stratification.
\end{OP}

\section{A glimpse of $\A_g$ in finite characteristic}\label{charp}
In this section we very briefly list the differences between the
results over $\C$ that we discussed so far, and the case of the base
field of finite characteristic. There is vast literature, and lots
of other interesting questions on $\A_g$ in finite characteristic
--- we refer to \cite{vdgoo}, \cite{oort1}, \cite{oort2},
\cite{vdgbook} for more details, reviews, and further references.
Here we just list what happens --- from now on we are always talking
about characteristic $p$.

\smallskip
The concept of a ppav is still defined, and the algebraic definition
of the moduli space $\A_g$ still makes sense. However, the universal
cover of a ppav is no longer $\C^g$, and thus the discussion about
period matrices, lattices, the Siegel upper half-space and the
symplectic group action no longer applies. There is, however, a way
to define theta functions algebraically over any base field, though
not all the techniques used in working with holomorphic theta
functions are applicable.

The Satake and toroidal compactifications are defined over arbitrary
base fields; the theory of Siegel modular forms and induced
embeddings as we gave it is specific to the base field $\C$, but
there is a concept of modular forms in finite characteristic.

The results on the nef cones of $\A_g^*$ and $\AP$ hold in any
characteristic. However, the resolution of singularities in finite
characteristic is not known, and the minimal model program is not
established, so we cannot speak of the canonical models anymore.
Neither does the discussion of effective divisors and Kodaira
dimension/general type issues carry over to the case of finite
characteristic.

\smallskip
The study of subvarieties of $\A_g$ in finite characteristic is
entirely different. Recall that over $\C$ by theorem \ref{oort}
$\A_g$ does not have a closed subvariety of codimension $g$, in
stark contrast to the following.
\begin{thm}[Koblitz \cite{koblitz} --- dimension, Oort \cite{oort0} --- completeness]
In finite characteristic the moduli space $\A_g$ has a complete
subvariety of codimension $g$ --- the locus of ppavs that do not
have points of order $p$ different from zero.
\end{thm}
\begin{df}
This observation is in a sense a byproduct of the study of the
powerful {\it Ekedahl-Oort stratification} \cite{oort1} of $\A_g$.
What one does is consider the group scheme $A[p]$ of points
of order $p$ on a ppav, with the symplectic pairing on it induced by the principal
polarization on $A$. One then defines the {\it Ekedahl-Oort stratum} as the
locus of ppavs $A$ for which the group scheme $A[p]$ is of a given
type, up to an isomorphism. It can be shown with a lot of work that there are finitely
many strata, each of which is quasi-affine, so that this
stratification gives a cell decomposition of $\A_g$.
\end{df}
Let $k$ be an algebraically closed field with ${\rm char\,} k= p$.
We define the {\it $p$-rank} of a ppav $A$ to be $f:=\log_p\sharp
A[p](k)$. Let $V_f$ be the locus of abelian varieties of $p$-rank at
most $f$.
\begin{thm}[van der Geer \cite{vdgeer1}]
The cycle class of the locus of ppavs of $p$-rank $\le f$ is
$$
 [V_f]=(p-1)(p^2-1)\cdots(p^{g-f}-1)\ll_{g-f},
$$
so in finite characteristic the Hodge classes are effectively
represented by subvarieties (not complete for $f>0$) of $\A_g$.
\end{thm}
It turns out that in fact all cycle classes of the Ekedahl-Oort
stratification lie in the tautological ring and can be computed
explicitly.

There also exists another stratification of $\A_g$ in finite
characteristic, by Newton polygon --- see \cite{oort3} for recent
work on it. There is a multitude of other constructions, results,
and questions concerning $\A_g$ in finite characteristic, which we
do not discuss here. The forthcoming book \cite{vdgbook} will
be a great source of information on moduli spaces of abelian
varieties in finite characteristics.


\end{document}